\documentclass[12pt]{amsart}
\usepackage{amscd,amssymb}
\usepackage{graphicx}
\usepackage[arrow,matrix]{xy}
\usepackage{pst-all}

\theoremstyle{plain}
\newtheorem{thm}[subsection]{Theorem}

\newtheorem{prop}[subsection]{Proposition}
\newtheorem{cor}[subsection]{Corollary}

\newtheorem{lemma}[subsection]{Mather's Lemma}

\theoremstyle{definition}
\newtheorem{rk}[subsection]{Remark}
\newtheorem{defn}[subsection]{Definition}

\numberwithin{equation}{section} 

\newcommand{\R}{\mathbb{R}}
\newcommand{\C}{\mathbb{C}}

\newcommand{\ir}{\mathcal{R}}

\newcommand{\ik}{\mathcal{K}}
\newcommand{\jj}{\mathcal{J}}

\begin{document}

\title[Homogeneous polynomials with isomorphic Milnor algebras]{Homogeneous polynomials with isomorphic Milnor algebras}
\author[Imran Ahmed]{Imran Ahmed}

\bigskip
\medskip
\address{Abdus Salam School of Mathematical Sciences,
         Government College University,
         68-B New Muslim Town Lahore,
         PAKISTAN.}
\address{COMSATS Institute of Information Technology, M.A. Jinnah Campus, Defence Road, off Raiwind Road Lahore,
         PAKISTAN.}
\email {iahmedthegr8@gmail.com}



\maketitle

%\tableofcontents

\section{Introduction}

In this note we recall first Mather's Lemma \ref{lem4.3.1} providing
effective necessary and sufficient conditions for a connected
submanifold to be contained in an orbit. In Theorem \ref{th8.1} we show that two homogeneous
polynomials $f$ and $g$ having isomorphic Milnor algebras are right-equivalent.

This is similar
to the celebrated theorem by Mather and Yau \cite{MY}, saying that
the isolated hypersurface singularities are determined by their
Tjurina algebras. Our result applies only to homogeneous
polynomials, but it is no longer necessary to impose the condition
of having isolated singularities at the origin.

\section{Preliminary results}
We recall here some basic facts on semialgebraic sets , which are
also called constructible sets, especially in the complex case. For
a more complete introduction we can see \cite{GWPL}, chapter 1.

\begin{defn}\label{def4.1.2}
Let $M$ be a smooth algebraic variety over $K$ ($K=\R$ or $\C$ as
usual.)\\
(i) {\bf Complex Case.} A subset $A\subset M$ is called
semialgebraic if $A$ belongs to the Boolean subalgebra generated by
the Zariski closed subsets of $M$ in the Boolean algebra $P(M)$ of
all subsets of $M$.\\\\
(ii) {\bf Real Case}. A subset $A\subset M$ is called semialgebraic
if $A$ belongs to Boolean subalgebra generated by the open sets
$U_f=\{x\in U;\,f(x)>0\}$ where $U\subset M$ is an algebraic open
subset in $M$ and $f:M\to \R$ is an algebraic function, in the
Boolean algebra $P(M)$ of all subsets of $M$.
\end{defn}

By definition, it follows that the class of semialgebraic subsets of
$M$ is closed under finite unions, finite intersections and
complements. If $f:M\to N$ is an algebraic mapping among the smooth
algebraic varieties $M$ and $N$ and if $B\subset N$ is
semialgebraic, then clearly $f^{-1}(B)$ is semialgebraic in $M$.
Conversely, we have the following basic result.

\begin{thm}\label{th4.1.3} (TARSKI-SEIDENBERG-CHEVALLEY)\\
If $A\subset M$ is semialgebraic, then $f(A)\subset N$ is also
semialgebraic.
\end{thm}

Next consider the following useful result.

\begin{prop}\label{prop4.1.8}
Let $G$ be an algebraic group acting (algebraically) on a smooth
algebraic variety $M$. Then the corresponding orbits are smooth
semialgebraic subsets in $M$.
\end{prop}

Let $m:G\times M\to M$ be a smooth action. In order to decide
whether two elements $x_0,\,x_1\in M$ are $G$-transversal, we try to find a path (a homotopy) $P=\{x_t; t\in [0,1]\}$
such that $P$ is entirely contained in a $G$-orbit. It turns out
this naive approach works quite well and the next result gives
effective necessary  and sufficient conditions for a connected
submanifold (in our case the path $P$) to be contained in an orbit.

\begin{lemma} (\cite{Ma})\label{lem4.3.1}
Let $m:G\times M\to M$ be a smooth action and $P\subset M$ a
connected smooth submanifold. Then $P$ is contained in a single
$G$-orbit if and only if the following conditions are
fulfilled:\\
(a) $T_x(G.x)\supset T_xP$, for any $x \in P$.\\
(b) $\dim T_x(G.x)$ is constant for $x\in P$.
\end{lemma}

\section{Main Theorem}

For isolated hypersurface singularities, the following result was
obtained by Mather and Yau, see \cite{MY}, 1982.

\begin{thm}\label{thMY}
Let $f,g:(\C^n,0) \to (\C,0)$ be two isolated hypersurface
singularities having isomorphic Tjurina algebras $T(f) \simeq T(g)$.
Then $f\stackrel{\ik}{\sim}g$, where $\stackrel{\ik}{\sim}$ denotes
the contact  equivalence.
\end{thm}

For arbitrary (i.e. not necessary with isolated singularities)
homogeneous polynomials we establish now the following result.

\begin{thm}\label{th8.1}
Let $f,g\in\C[x_1,\ldots,x_n]_d=H^d(n,1;\C)=H^d$ be two homogeneous
polynomials of degree $d$ such that $\jj_f=\jj_g$. Then
$f\stackrel{\ir}{\sim}g$, where $\stackrel{\ir}{\sim}$ denotes the
right equivalence.
\end{thm}

\begin{proof}
To prove this claim choose an appropriate submanifold of
$H^d(n,1;\C)$ containing $f$ and $g$ and then apply Mather's lemma
to get the result.\\ Here $f,g\in H^d(n,1;\C)$ such that
$\jj_f=\jj_g$. Set $f_t=(1-t)f+tg\in H^d(n,1;\C)$. Consider the
$\ir$-equivalence action on $H^d(n,1;\C)$ under the group
$Gl(n,\C)$. By eq. (4.16) \cite{D1}, p35, we have
\begin{equation}\label{eq8.1}
T_{f_t}(Gl(n,\C).f_t)=\C<x_j\frac{\partial f_t}{\partial
x_i};\,\,i,j=1,\ldots,n>
\end{equation}
Now, note that the R.H.S of eq. \eqref{eq8.1} satisfies the relation
\[
\C<x_j\frac{\partial f_t}{\partial
x_i};\,\,i,j=1,\ldots,n>\subset\jj_{f_t}\cap H^d
\]
But $\jj_{f_t}\cap H^d\subset \jj_f\cap H^d$ since
\[\frac{\partial f_t}{\partial x_i}=(1-t)\frac{\partial f}{\partial
x_i}+t\frac{\partial g}{\partial
x_i}\in(1-t)\jj_f+t\jj_g=\jj_f\,\,(\because \jj_f=\jj_g)\] So, we
have the inclusion of finite dimensional $\C$-vector spaces
\begin{equation}\label{eq8.2}
T_{f_t}(Gl(n,\C).f_t)=\C<x_j\frac{\partial f_t}{\partial
x_i};\,\,i,j=1,\ldots,n>\subset \jj_f\cap H^d
\end{equation}
with equality for $t=0$ and $t=1$.\\Let's  show that we have
equality for all $t\in[0,1]$ except finitely many values.\\
Clearly the dimension of the space $\jj_f\cap H^d$ is at most $n^2$.
Let's fix $\{e_1,\ldots,e_m\}$ a basis of $\jj_f\cap H^d$, where
$m\leq n^2$. Consider the $n^2$ polynomials
\[\alpha_{ij}(t)=x_j\frac{\partial f_t}{\partial
x_i}=x_j[(1-t)\frac{\partial f}{\partial x_i}+t\frac{\partial
g}{\partial x_i}]\] corresponds to the space \eqref{eq8.1}. We can
express each $\alpha_{ij}(t)$, $i,j=1,\ldots,n$ in terms of above
mentioned fixed basis as
\begin{equation}\label{eq8.3}
\alpha_{ij}(t)=\phi_{ij}^1(t)e_1+\ldots+\phi_{ij}^m(t)e_m,\,\,\forall\,\,
i,j=1,\ldots,n
\end{equation}
where each $\phi_{ij}^k(t)$ is linear in $t$. Consider the matrix of
transformation corresponding to the eqs. \eqref{eq8.3}
\[
(\phi_{ij}^m(t))_{n^2\times m}= \left(
  \begin{array}{cccc}
    \phi_{11}^1(t) & \phi_{11}^2(t)& \ldots & \phi_{11}^m(t) \\
    \vdots & \vdots & \ddots & \vdots \\
    \phi_{1n}^1(t) & \phi_{1n}^2(t) & \ldots & \phi_{1n}^m(t) \\
    \phi_{21}^1(t) & \phi_{21}^2(t) & \ldots & \phi_{21}^m(t) \\
    \vdots & \vdots & \ddots & \vdots \\
    \phi_{nn}^1(t) & \phi_{nn}^2(t) & \ldots & \phi_{nn}^m(t) \\
  \end{array}
\right)
\]
having rank at most $m$. Note that the equality $<x_j\frac{\partial
f_t}{\partial x_i}>=\jj_f\cap H^d$ holds for those values of $t$ in
$\C$ for which the rank of above matrix is precisely $m$. As we know
that the rank of the matrix of transformation is at most $m$,
therefore there must be $n^2-m$ proportional rows. So, we have the
$m\times m$-sub matrix whose determinant is a polynomial of degree
$m$ in $t$ and by the fundamental theorem of algebra it has at most
$m$ roots in $\C$ for which rank of the matrix of transformation
will be less than $m$. Therefore, the above-mentioned equality does
not hold for a finitely many values, say $t_1,\ldots,t_p$ where
$1\leq p\leq m$.\\
It follows that the dimension of the space \eqref{eq8.1} is constant
for all $t\in \C$ except finitely many values $\{t_1,\ldots,t_p\}$.\\
For an arbitrary smooth path
\[\alpha:\C\longrightarrow\C\backslash \{t_1,\ldots,t_p\}\]
with $\alpha(0)=0$ and $\alpha(1)=1$, we have the connected smooth
submanifold
\[P=\{f_t=(1-\alpha(t))f(x)+\alpha(t)g(x):\,t\in\C\}\]
of $H^d$. By the above, it follows $\dim T_{f_t}((Gl(n,\C).f_t))$ is
constant for $f_t\in P$.\\
Now, to apply Mather's lemma, we need to show that the tangent space
to the submanifold $P$ is contained in that to the orbit
$Gl(n,\C).f_t$ for any $f_t\in P$. One clearly has
\[T_{f_t}P=\{\dot{f_t}=-\dot{\alpha}(t)f(x)+\dot{\alpha}(t)g(x):\,\forall\,t\in\C\}\]
Therefore, by Euler formula 7.6 \cite{D1}, p 101, we have\[T_{f_t}P\subset
T_{f_t}(Gl(n,\C).f_t)\] By Mather's lemma the submanifold $P$ is
contained in a single orbit. It implies that
$f\stackrel{\ir}{\sim}g$, as required.
\end{proof}

\begin{cor}\label{cor8.2}
Let $f,\,g\in H^d(n,1;\C)$. If $M(f)\simeq M(g)$ (isomorphism of
graded $\C$-algebra) then $f\stackrel{\ir}{\sim}g$.
\end{cor}

\begin{proof}
We show firstly that an isomorphism of graded $\C$-algebras
\[\varphi:M(g)=\frac{\C[x_1,\ldots,x_n]}{\jj_g}\stackrel{\simeq}{\longrightarrow}
M(f)=\frac{\C[x_1,\ldots,x_n]}{\jj_f}\] is induced by a linear
isomorphism $u:\C^n\longrightarrow\C^n$ such that
$u^*(\jj_g)=\jj_f$.\\Consider the following commutative diagram.
$$\xymatrix{
0 \ar[d] & 0 \ar[d] \\
\jj_g  \ar@{.>}[r]^{u^*} \ar[d]^i & \jj_f \ar[d]^j \\
\C[x_1,\ldots,x_n]  \ar@{.>}[r]^{u^*} \ar[d]^p & \C[x_1,\ldots,x_n]
\ar[d]^q
\\
M(g) \ar[r]^{\varphi}_{\simeq} \ar[d] & M(f) \ar[d] \\
0 & 0 }
$$

We note that the isomorphism $\varphi$ is a degree preserving map
and each of the Jacobian ideals $\jj_f$ and $\jj_g$ is generated by
the homogeneous polynomials of degree $d-1$. The cases $d=1$ and
$d=2$ are special, and we can treat them easily. Assume from now on
that $d\geq 3$. Define the morphism
$u^*:\C[x_1,\ldots,x_n]\rightarrow \C[x_1,\ldots,x_n]$ by
\begin{equation}\label{eq8.4}
u^*(x_i)=Li(x_1,\ldots,x_n)=a_{i1}x_1+\ldots+a_{in}x_n\,,\,\,a_{i1},\ldots,a_{in}\in\C
\end{equation}
which is well defined by commutativity of diagram below.
$$\xymatrix{
x_i  \ar@{|.>}[r]^{u^*} \ar@{|->}[d]^p &  Li\ar@{|->}[d]^q \\
\widehat{x_i} \ar@{|->}[r]^{\varphi}_{\simeq} & \widehat{L_i}}$$

Let's prove that $u^*$ is an isomorphism. Consider the following
commutative diagram at the level of degree $d=1$.
$$\xymatrix{
\C[x_1,\ldots,x_n]_1  \ar@{.>}[r]^{u_1^*} \ar@{=}[d]^p &
\C[x_1,\ldots,x_n]_1 \ar@{=}[d]^q \\
(M(g))_1 \ar[r]^{\varphi_1}_{\simeq} & (M(f))_1}
$$

Since here the Jacobian ideals $\jj_f$ and $\jj_g$ are generated by
polynomials of degree $\geq 2$, therefore we
have\[(M(g))_1=(M(f))_1=\C[x_1,\ldots,x_n]_1\] It implies that
$\varphi_1$ and $u_1^*$ are coincident. As $\varphi$ is a given
graded $\C$-algebra isomorphism, it follows that $u_1^*$ is also an
isomorphism. Hence $u^*$ is an isomorphism.
\\Next we show that $u^*(\jj_g)=\jj_f$.\\ For every
$G\in\jj_g$, we have $u^*(G)\in\jj_f$ by commutative diagram below.
$$\xymatrix{
G  \ar@{|->}[r]^{u^*} \ar@{|->}[d]^p &  F=u^*(G) \ar@{|->}[d]^q \\
\widehat{0} \ar@{|->}[r]^{\varphi} & \widehat{F}=\widehat{0}}
$$

It implies that $u^*(\jj_g)\subset\jj_f$. As $u^*$ is an
isomorphism, therefore it is invertible and by repeating the above
argument for its inverse, we have $u^*(\jj_g)\supset\jj_f$.\\Thus,
$u^*$ is an isomorphism with $u^*(\jj_g)=\jj_f$.\\
By eq. \eqref{eq8.4}, the map $u:\C^n\to\C^n$ can be defined by
$$u(z_1,\ldots,z_n)=(L_1(z_1,\ldots,z_n),\ldots,L_n(z_1,\ldots,z_n))$$
where $L_i(z_1,\ldots,z_n)=a_{i1}z_1+\ldots+a_{in}z_n$.\\Note that
that $u$ is a linear isomorphism by Prop. 3.16 \cite{D1}, p.23.\\
In this way, we have shown that the isomorphism $\varphi$ is induced
by a linear isomorphism $u:\C^n\to\C^n$ such that $u^*(\jj_g)=\jj_f$.\\
Here $u^*(\jj_g)=<g_1\circ u,\ldots,g_n\circ u>=\jj_{g\circ u}$,
where $g_j=\frac{\partial g}{\partial x_j}$.\\
Therefore, $\jj_{g\circ u}=\jj_f\Rightarrow g\circ
u\stackrel{\ir}{\sim}f$. Now, $g\circ u\stackrel{\ir}{\sim}g$
follows that $g\stackrel{\ir}{\sim}f$.
\end{proof}

\begin{rk}\label{rk8.3}
The converse implication, namely
\[f\stackrel{\ir}{\sim}g\Rightarrow M(f)\simeq M(g)\]
always holds(even for analytic germs $f,\,g$ defining IHS), see
\cite{D1}, p90.
\end{rk}

\end{document}